\documentclass[VANCOUVER,STIX1COL]{WileyNJD-v2}
\articletype{Article}

\received{}
\revised{}
\accepted{}

\usepackage{booktabs}
\usepackage{url}

\def\vect#1{{\bf #1}}
\def\op#1{{\mathcal #1}}
\newcommand{\newt}[2]{{\mathcal N}_{#1}(#2)}
\newcommand{\Cd}{\mathfrak{C}}
\newcommand{\bjconst}{\mathfrak{K}}
\newcommand{\bjfrac}{\mathfrak{f}}
\newcommand{\nlevels}{\ell}
\newcommand{\ident}{I}
\newcommand{\functional}{J}
\newcommand{\redfunctional}{\hat{J}}
\newcommand{\reg}{\beta}
\newcommand{\cstate}{y}

\newcommand{\statespace}{Y}
\newcommand{\statedim}{m}
\newcommand{\statedimcoarse}{M}
\newcommand{\statebas}{\varphi}
\newcommand{\statebascoarse}{\Phi}
\newcommand{\ctrldim}{n}
\newcommand{\ctrldimcoarse}{N}
\newcommand{\ctrlbas}{\psi}
\newcommand{\ctrlbascoarse}{\Psi}
\newcommand{\desired}{y_d}
\newcommand{\control}{u}
\newcommand{\bilform}{a}
\newcommand{\test}{v}
\newcommand{\controlspace}{U}
\newcommand{\stiffness}{{\bf A}}
\newcommand{\mass}{{\bf M}}
\newcommand{\massstate}{\mass_{\cstate}}
\newcommand{\masscontrol}{\mass_{\control}}
\newcommand{\massstatecontrol}{\mass_{\cstate \control}}

\newcommand{\solutionop}{ {\bf K}}
\newcommand{\solutionoptrue}{\mathcal K}
\newcommand{\soloperatornormbound}{C_{\mathcal K}}
\newcommand{\hessian}{{\bf  G}}
\newcommand{\hessianop}{\mathcal G}
\newcommand{\twogrid}{{\bf T}}

\newcommand{\multigrid}{{\bf W}}
\newcommand{\multigridop}{{\mathcal W}}

\newcommand{\stateinterp}{{\bf S}}
\newcommand{\controlinterp}{{\bf P}}
\newcommand{\controlprojectop}{ \Pi}
\newcommand{\controlproject}{{\bf \Pi}}
\newcommand{\stateinterpop}{{\mathcal S}}
\newcommand{\controlinterpop}{{\mathcal P}}

\newcommand{\R}{{\mathbb R}}

\newcommand{\domain}{\Omega}
\newcommand{\coeff}{\kappa}
\newcommand{\contrast}{\alpha}
\newcommand{\specdist}{d_{\sigma}}

\begin{document}

\title{Algebraic multigrid preconditioning of the Hessian in PDE-constrained
  optimization\protect\thanks{The work of the first author was performed under the auspices of
    the U.S.~Department of Energy under Contract DE-AC52-07NA27344
    (LLNL-JRNL-802278). The work of the second author is supported by the
    U.S. Department of Energy
    Office of Science, Office of Advanced  Scientific Computing
    Research, Applied Mathematics program under Award Number
    DE-SC0005455, and by the National Science Foundation under award
    DMS-1913201.}}

\author[1]{Andrew T. Barker*}
\author[2]{Andrei Dr{\u a}g{\u a}nescu}

\authormark{Barker and Dr{\u a}g{\u a}nescu}

\address[1]{\orgdiv{Center for Applied Scientific Computing}, \orgname{Lawrence Livermore National Laboratory}, \orgaddress{\state{California}, \country{United States}}}
\address[2]{\orgname{University of Maryland, Baltimore County}, \orgaddress{\state{Maryland}, \country{United States}}}

\corres{* Andrew T. Barker, \email{atb@llnl.gov}}

\abstract[Abstract]{We construct an algebraic multigrid (AMG) based preconditioner for the 
reduced Hessian of a linear-quadratic optimization problem
constrained by an elliptic partial differential equation. While the preconditioner generalizes
a geometric multigrid preconditioner introduced in earlier works,
its construction relies entirely on a standard AMG infrastructure 
built for solving the forward elliptic equation, thus allowing for it 
to be implemented using a variety of AMG methods and standard packages. 
Our analysis establishes a clear connection between the quality of the preconditioner 
and the AMG method used. The proposed strategy has a broad and robust applicability to problems 
with unstructured grids, complex geometry,
and varying coefficients. The method is implemented using the Hypre package and several
numerical examples are presented.}

\keywords{algebraic multigrid, PDE-constrained optimization, elliptic equations, finite element methods}

\maketitle

\section{Introduction}

The purpose of this work is to construct and analyze  algebraic multigrid (AMG)
based preconditioners for the reduced Hessian for optimal control problems constrained by
elliptic partial differential equations (PDEs).
In this setting, an {\em unpreconditioned} conjugate gradient algorithm
already can be shown to converge independently of the mesh size $h$ (see~\cite{Engl_Hanke_Neubauer} Ch. 7, 
also~\cite{draganescu-dupont}).
A preconditioner introduced in \cite{draganescu-dupont} can even
improve on this, so that the number of iterations {\em decreases} as
the problem size grows.  Even though this system is easy to
precondition in some sense, each iteration is extremely expensive,
requiring a forward and an adjoint solve of the underlying PDE.  Since
the absolute number of iterations can be large, and each iteration is
expensive, in practice it may be desirable to have an efficient preconditioner to
reduce the number of iterations.

Multilevel preconditioners for this problem setting have been
discussed in \cite{biros-dogan-multilevel, hanke-vogel-two-level,
  draganescu-dupont, gong-xie-multilevel},  among many other
references, but to date there is a lack of robust practical
implementations for challenging cases including unstructured grids,
varying coefficients, and complicated geometries.  AMG
solvers and preconditioners have long been a well-developed strategy
for solving the forward problem in complicated practical settings.
The literature here is too large to even scratch the surface, but a
reasonable starting point is the recent review
\cite{xu-zikatanov-amg}. However, reduced Hessians for optimal control
problems constrained by PDEs are not good candidates for the direct application 
of AMG methods, primarily because they come in the form of dense matrices with no
obvious sparse aproximations, which is due to the fact that they represent integral operators
that are non-local. Our goal in this paper is show that, even though standard AMG methodology is not
applicable, the AMG framework provides all the elements needed for preconditioning the reduced Hessian. 

In particular, we implement an extension of the multilevel framework
of \cite{draganescu-dupont} that allows the use of an algebraic
multigrid hierarchy in place of  the geometric multigrid hierarchy.
The preconditioner for the reduced Hessian can be constructed
systematically based on the interpolation and restriction operators
for the forward problem, which are readily available in most algebraic
multigrid software implementations.  The convergence theory depends on
the underlying approximation properties of the multigrid hierarchy.
In a standard algebraic multigrid context, these approximation
properties {may} not allow for the same improvement
with grid refinement that we see in the geometric multigrid setting,
but the approach below is flexible and allows for a Hessian
preconditioner for any multilevel discretization of the underlying
forward problem.  In particular, if an algebraic hierarchy with
appropriate approximation properties is available for the forward
problem, the preconditioning approach below will recover the
appropriate convergence for the PDE-constrained optimization problem.

% A more typical approach to PDE-constrained optimization uses a quasi-Newton approximation to the reduced Hessian.
% When a full-Newton approach is used, the quasi-Newton reduced Hessian is often used as a preconditioner \cite{biros-ghattas-lnks}.
% The preconditioning strategy proposed here for the reduced Hessian is different.

The numerical results show that the algebraic variant can be quite
effective in reducing iteration counts even when the theory does not
apply cleanly.  Since algebraic preconditioners allow easier
application to unstructured grids and problems with varying
coefficients, the approach is practically useful in a wide variety of
situations even apart from the convergence theory.

Our PDE-constrained optimization setting is introduced in Section
\ref{sec:problem}, followed by a description of the practical
algorithm in Section \ref{sec:preconditioning}.  A theoretical framework is developed in \ref{sec:analysis}, followed
by numerical results in Section \ref{sec:numerics}. Some conclusions are formulated in Section~\ref{sec:conclusions}.

\section{The optimization problem}
\label{sec:problem}
To fix ideas we focus on the standard distributed elliptic-constrained problem
  \begin{equation}
    %\mathrm{minimize}
    \min_{y,u}\ \  \functional(y,u) = \frac{1}{2} \| \cstate - \desired \|^2 + \frac{\reg}{2} \| \control \|^2
  \label{eq:functionalcont}
\end{equation}
subject to 
\begin{equation}
  -\nabla \cdot (\coeff \nabla \cstate) = \control\ \ \mathrm{in}\  \Omega,\ \ y|_{\partial \Omega} = 0,
  \label{eq:pdecont}
\end{equation}
where $\Omega\subset \R^d$ ($d=2, 3$) is a sufficiently regular bounded domain, $\desired \in L^2(\Omega)$ 
is a desired state, and $\|\cdot\|$ is the $L^2(\Omega)$-norm. 
The coefficient $ \coeff(x) \in \R $ is assumed to satisfy $ 0 < \coeff_0 \leq \coeff(x) \leq \coeff_1 < \infty $ for all $x\in \Omega$, for some
positive constants $\coeff_0, \coeff_1$.
We refer to $\control$ as the control and
to $\cstate$ as the state. The  problem of interest is a discrete counterpart of~\eqref{eq:functionalcont}--\eqref{eq:pdecont} resulted from 
discretizing~\eqref{eq:pdecont} using a standard finite element formulation.
%arising from a finite element discretization
%of the PDE~\eqref{eq:pdecont}, and a finite element representation of the control. More precisely, 
To that end consider a finite element space $\statespace\subset H_0^1(\Omega)$ with basis
$(\statebas_i)_{1\le i\le \statedim}$  and  a discrete control space $\controlspace\subset L^2(\Omega)$  with basis $(\ctrlbas_i)_{1\le i\le \ctrldim}$.
The standard discrete Galerkin  formulation of~\eqref{eq:pdecont} is: 
\begin{equation}
  \mathrm{Find} \ \cstate\in \statespace\ \mathrm{so\  that}\ \ \bilform(\cstate,\test) = (\control,\test)\ \ \forall \test \in \statespace,
  \label{eq:pdecontweak}
\end{equation}
where
\begin{equation}
  \bilform(\cstate,\test) = \int_{\Omega} \coeff \nabla \cstate\cdot \nabla \test\ \ \forall \cstate, \test \in H_0^1(\Omega),
  \label{eq:bilform}
\end{equation}
and $(\cdot,\cdot)$ is the standard $L^2$-inner product. We introduce the solution operator $\solutionoptrue\in \mathfrak{L}(\controlspace, \statespace)$ defined by
\begin{equation*}
\solutionoptrue \control=\cstate\ \ \mathrm{if}\ \ y\ \ \mathrm{satisfies}\ \eqref{eq:pdecontweak}.
\end{equation*}

Since we prefer to distinguish between operators (acting on the finite element spaces) and their matrix representations, 
we will denote operators with caligraphic font and matrices using bold font. Furthermore,  
for a discrete function we use bold notation
to denote the vector of coordinates with respect to the basis introduced in the 
space where the function resides. For example, if ${y}\in \statespace$, then 
$\vect{y}=[y_1, y_2, \dots, y_{\statedim}]^T\in\R^{\statedim}$ is defined so that $y=\sum_{i=1}^{\statedim} y_i \statebas_i$. 
The matrices needed for the discrete control problem are the \emph{stiffness matrix}
$[\stiffness]_{ij} = \bilform(\statebas_j,\statebas_i)$, the \emph{state mass matrix} $[\massstate]_{i j}=(\statebas_j,\statebas_i)$, 
the \emph{control mass matrix}
$[\masscontrol]_{i j}=(\ctrlbas_j,\ctrlbas_i)$, and the \emph{control-to-state mass matrix} 
$[\massstatecontrol]_{i,j}=(\statebas_i,\ctrlbas_j)$.
Note that
both $\statespace$ and $\controlspace$ inherit the inner-product from $L^2(\Omega)$, and that
$\|y\|^2=\vect{y}^T \massstate \vect{y}$ for $y\in \statespace$, and $\|u\|^2=\vect{u}^T \masscontrol \vect{u}$ for $u\in \controlspace$.
The actual discrete problem to be solved is
% Find a control $\control\in \controlspace$ to minimize the functional
\begin{equation}
  \min_{y,u}\functional(y,u) = \frac{1}{2} \| \cstate - \desired \|^2 + \frac{\reg}{2} \| \control \|^2 = 
  \frac{1}{2} (\vect{\cstate} - \vect{\desired})^T \massstate (\vect{\cstate} - \vect{\desired}) + \frac{\reg}{2}\vect{\control}^T \masscontrol \vect{\control}
  \label{eq:functional}
\end{equation}
subject to the constraint
\begin{equation}
  \stiffness \vect{\cstate} = \massstatecontrol {\vect{\control}}.
  \label{eq:state}
\end{equation}
%In our discrete setting we norm $ \statespace $ and $ \controlspace $ with their mass matrices,
%\begin{equation}
%  \| \cstate \|_{\statespace}^2 = \cstate^T \massstate \cstate, \quad
%  \| \control \|_{\controlspace}^2 = \control^T \masscontrol \control.
%\end{equation}

We should remark that the problem~\eqref{eq:functionalcont}--\eqref{eq:pdecont} and its discretization~\eqref{eq:functional}--\eqref{eq:state}
is  among the most commonly studied in the PDE-constrained optimization literature~\cite{Troltzsch-textbook}; therefore it is a natural example
to showcase our method. However, the technique presented
in this paper can be directly applied to a variety of other linear-quadratic optimal control problems constrained by PDEs, 
such as boundary control of elliptic equations, 
initial value control and/or space-time distributed optimal control of parabolic equations, etc, as long as their discretizations can be expressed 
as~\eqref{eq:functional}--\eqref{eq:state}. As with  the geometric form of the multigrid algorithm, 
the performance of the presented method  will vary from one model problem to another.

% We formulate our optimization problem in a discrete setting, where we have a state space
% $ \statespace $ and a control space $ \controlspace $ which we think of as finite element discretizations of some appropriate function spaces.
% Our discretization also includes a discretized PDE operator 
% $ \stiffness : \statespace \rightarrow \statespace $ (which we can think of for now as a finite element discretization of the Laplacian) and several mass matrices, 
% $ \massstate : \statespace \rightarrow \statespace, \masscontrol : \controlspace \rightarrow \controlspace $, and $ \massstatecontrol : \controlspace \rightarrow \statespace $.
% (In practice, $\masscontrol$ has no boundary conditions eliminated while $\massstate$ does.)

We reformulate the optimal control problem \eqref{eq:functional}--\eqref{eq:state} as an unconstrained problem
 by defining a solution matrix for~\eqref{eq:state}
\begin{equation}
  \label{eq:solutionopdef}
  \solutionop = \stiffness^{-1} \massstatecontrol ,
\end{equation}
which is precisely the matrix representation of the  operator  $\solutionoptrue$. 
Using $\solutionop$ we eliminate the state $ \cstate $ (or $\vect{y}$) in \eqref{eq:functional} and obtain the first order optimality condition by
setting the gradient of $\redfunctional(\vect{\control}) = \functional(\solutionop\vect{\control},\vect{\control})$  
to zero, followed by left-multiplying with $\masscontrol^{-1}$:
\begin{equation}
  \left(\masscontrol^{-1} \solutionop^T \massstate \solutionop + \reg \vect{I} \right) \vect{\control} =  \masscontrol^{-1}\solutionop^T \massstate \vect{\desired}.
  \label{eq:firstorder}
\end{equation}
Equation~\eqref{eq:firstorder} has a more simplified (and meaningful) form when using adjoint operators. By definition, $\solutionoptrue^\ast$ and its matrix $\solutionop^\ast$ 
satisfy
\[
\vect{\cstate}^T\massstate\solutionop \vect{\control}  =(\solutionoptrue \control,  \cstate)=(\control, \solutionoptrue^\ast \cstate) = (\solutionop^\ast\vect{\cstate})^T\masscontrol \vect{\control}
= \vect{\cstate}^T (\solutionop^\ast)^T\masscontrol \vect{\control} \ \ \ \forall \control\in \controlspace, \cstate\in \statespace.
\]
Hence, we obtain
\[
  \solutionop^\ast = \masscontrol^{-1} \solutionop^T \massstate = \masscontrol^{-1} \massstatecontrol^T \stiffness^{-T} \massstate.
\]
This allows us to  rewrite \eqref{eq:firstorder} in the familiar form 
%\begin{equation}
%  \left( \solutionoptrue^\ast \solutionoptrue + \reg I \right) \control = \solutionoptrue^\ast \desired.
%  \label{eq:truefirstorderop}
%\end{equation}
%and equivalent matrix form
\begin{equation}
  \hessian\vect{\control} \eqdef \left( \solutionop^\ast \solutionop + \reg \vect{I} \right) \vect{\control} = \solutionop^\ast \vect{\desired}.
  \label{eq:truefirstorder}
\end{equation}
The matrix $\hessian$ on the left side of~\eqref{eq:truefirstorder} is the Hessian of the reduced cost function $\redfunctional$, 
usually referred to as the reduced Hessian.
%\begin{equation*}
%  \hessian = \left( \solutionoptrue^\ast \solutionoptrue + \reg I \right),
%\end{equation*}
The goal of this paper is to construct multilevel preconditioners for $ \hessian $.

In general,  $ \hessian $ is dense and for large- and even medium-scale problems it would be extremely expensive, 
perhaps impossible, to form $ \hessian $ explicitly.
We can apply it (and even this operation is fairly expensive), but we cannot use its entries to construct an AMG 
preconditioner in the usual way.
In what follows we show how to use the matrices required (and available) for building an
AMG preconditioner for the stiffness matrix $ \stiffness $ in order to build a preconditioning algorithm for $ \hessian $.

\section{Preconditioning the Hessian}
\label{sec:preconditioning}

Our construction follows closely the algorithm in \cite{draganescu-dupont}. We first present the construction of the 
two-level preconditioner which then leads us to the multilevel version.

\subsection{Two-level preconditioner}
\label{ssec:twolevel}
In the  two-level setting, we define a coarse state space $ \statespace_H\subset \statespace$ and a coarse control spaces $ \controlspace_H \subset \controlspace$.  
To fix ideas we focus on a specific form of AMG, namely smoothed 
aggregation~\cite{vanek-mandel-sa, brezina-vanek-improved-sa}, where the coarse basis functions for both
state and controls are defined by prolongator matrices $ \stateinterp $ and $ \controlinterp $, respectively:
\[
\statebascoarse_k = \sum_{j=1}^{\statedim}[\stateinterp]_{j k}  \statebas_j,\ \ 1\le k\le \statedimcoarse,\ \ \mathrm{and}\ \ 
%\]
%\[
\ctrlbascoarse_k = \sum_{j=1}^{\ctrldim}[\controlinterp]_{j k}  \ctrlbas_j,\ \ 1\le k\le \ctrldimcoarse.
\]
The coarse state space $\statespace_H$ is simply the span of $(\statebascoarse_k)_{1\le k \le \statedimcoarse}$, and  the 
coarse control space $\controlspace_H$ is the span of $(\ctrlbascoarse_k)_{1\le k \le \ctrldimcoarse}$.
Since $\statebascoarse_k\in \statespace$ for  $1\le k\le \statedimcoarse$, and 
$\ctrlbascoarse_k\in \controlspace$ $1\le k\le \ctrldimcoarse$ we indeed have that
$ \statespace_H \subset \statespace, \controlspace_H \subset \controlspace $.
For future reference we remark that the prolongator matrices represent the embedding operators 
$ \stateinterpop\in\mathfrak{L}(\statespace_H, \statespace) $ and 
$ \controlinterpop\in\mathfrak{L}(\controlspace_H,\controlspace) $.
The specific form of the prolongator is not relevant at this point.

At the coarse level 
we formulate the discrete PDE by replacing $\statespace$ with $\statespace_H$ in~\eqref{eq:pdecontweak}, thus obtaining
a coarse stiffness matrix
\[
[ \stiffness_H]_{k l} = \bilform(\statebascoarse_l,\statebascoarse_k) = \sum_{i,j=1}^{\statedim}  \bilform( [\stateinterp]_{j l} \statebas_j,[\stateinterp]_{i k}\statebas_i) = 
\sum_{i,j=1}^{\statedim} [\stateinterp^T]_{ki} [\stiffness]_{ij} [\stateinterp]_{jl} = [\stateinterp^T \stiffness \stateinterp]_{kl}.
\]
Therefore
\begin{align}
  \stiffness_H &= \stateinterp^T \stiffness \stateinterp   \label{eq:coarsestiffness}.
\end{align}
A similar calculation leads to the definition of the analogous coarse-level matrices
\begin{equation}
  \mass_{\cstate,H} = \stateinterp^T \massstate \stateinterp, \quad
  \mass_{\control,H} = \controlinterp^T \masscontrol \controlinterp, \quad
  \mass_{ \cstate\control, H} = \stateinterp^T \massstatecontrol \controlinterp . \label{eq:coarsemass}
\end{equation}
Hence, a coarse version of the state equation \eqref{eq:state} can now be written as
\[
  \stiffness_H  \vect{\cstate}_H = \mass_{ \cstate\control, H}  \vect{\control}_H
\]
and we can define a coarse solution matrix by
\[
  \solutionop_H = (\stateinterp^T \stiffness \stateinterp)^{-1} (\stateinterp^T \massstatecontrol \controlinterp) = (\stiffness_H)^{-1} \mass_{ \cstate\control, H},
\]
and a coarse Hessian
\[
\hessian_H = \solutionop^\ast_H \solutionop_H + \reg\vect{I},
\]
with $\solutionop^\ast_H =\mass_{\control,H}^{-1} \solutionop^T_H \mass_{\cstate,H}$.

%Hence
%For now we assume the existence of prolongators $ \stateinterp : \statespace_H
%\rightarrow \statespace $ and $ \controlinterp : \controlspace_H
%\rightarrow \controlspace $. 
%In a moment we will describe how to
%obtain such prolongators with algebraic multigrid, but first we
%describe how to implement a multilevel preconditioner for $ \hessian $
%if these operators are given.

To define the two-level preconditioner we need the $ L^2 $-projection $ { \controlprojectop} : \controlspace \rightarrow \controlspace_H $. 
Since $\controlprojectop$ is the adjoint of the embedding, the matrix representation of $\controlprojectop$ is
\[
  \controlproject = \controlinterp^{\ast} = \mass_{\control,H}^{-1} \controlinterp^T \masscontrol .
\]
Note that in classical multigrid, and in particular in AMG, the usual restriction matrix is $ \controlinterp^T $. One of the key
features of our approach  is to use  $\controlproject$ instead of $ \controlinterp^T $.
%which is not allowed to be the usual multigrid $ \controlinterp^T $.
%(In fact not allowing it to be $\controlinterp^T$ is one of the key features of our approach.)
Then the two-grid preconditioner is defined as in \cite[(4.1)]{draganescu-dupont} by
\begin{equation}
  \twogrid = \controlinterp (\hessian_H) \controlproject + \reg (\vect{I} - \controlinterp \controlproject).
  \label{eq:twogriddirect}
\end{equation}
Actually, in practice the inverse $\twogrid^{-1}$ is needed; it can be easily verified that
\begin{equation}
  \twogrid^{-1} = \controlinterp (\hessian_H)^{-1} \controlproject + \reg^{-1} (\vect{I} - \controlinterp \controlproject).
  \label{eq:twogrid}
\end{equation}
Notably, neither $\twogrid$ nor $\twogrid^{-1}$ is explicitly formed, but $\twogrid^{-1}$ can be practically applied to a vector on the fine grid, provided the action of the inverse of the coarse Hessian $\hessian_H$ is available. 
Also note that,
due to~\eqref{eq:coarsemass},
$\controlinterp \controlproject$ is a projection matrix having as range the coarse space; 
applying this matrix only involves inverting the coarse control mass matrix.
The definition~\eqref{eq:twogrid} has an additive Schwarz structure, with a coarse grid correction and a kind of ``smoother''.
Since the operator we are preconditioning is a ``smoothing'' operator (that is, it involves the inverse of an elliptic operator), no real smoothing is required, 
just the above projection. 

One note on symmetry: both the Hessian $\hessian$ and the preconditioner $\twogrid$ are symmetric with respect to the $L^2$-induced inner product on
$\R^{\ctrldim}$, that is, $\hessian=\hessian^{\ast}$ and $\twogrid=\twogrid^{\ast}$. This is not the same as saying they are symmetric matrices, but rather 
$\hessian=\masscontrol^{-1} \hessian^T \masscontrol$ and $\twogrid=\masscontrol^{-1} \twogrid^T \masscontrol$. Hence special care has to be taken when 
using (preconditioned) 
conjugate gradient to solve the system~\eqref{eq:truefirstorder},  a matter that is further discussed in Section~\ref{ssec:numerics}.
In addition, it is shown in~\cite{draganescu-dupont} that $\twogrid$ is positive definite, a property that is not automatically 
shared by the multigrid preconditioner introduced in Section~\ref{ssec:multilevel}.

\subsection{Multilevel preconditioner}
\label{ssec:multilevel}

The multilevel preconditioner is not a straightforward recursion of the two-level preconditioner.
Indeed, a short calculation shows that a
simple V-cycle recursion in \eqref{eq:twogrid} results in just a two-grid method using an even coarser grid, and does not yield the
desired optimality result in the geometric multigrid setting~\cite{draganescu-dupont}.

To define a multilevel method precisely, we need some additional
notation.  An $\nlevels$-level preconditioner involves a hierarchy of state and control
space (numbered in the AMG tradition from fine to coarse)
$ \statespace = \statespace_0\supseteq \statespace_1\supseteq \dots \supseteq  \statespace_{\nlevels-1}$ and 
$ \controlspace = \controlspace_0\supseteq \controlspace_1\supseteq \dots \supseteq  \controlspace_{\nlevels-1}$, together with 
prolongation matrices $ \stateinterp_j$ corresponding to the embedding operators $ \stateinterpop_j\in\mathfrak{L}(\statespace_{j+1},\statespace_{j})$, and 
$ \controlinterp_j$   associated with embedding operators $ \controlinterpop_j\in\mathfrak{L}(\controlspace_{j+1},\controlspace_{j})$.
Then coarse stiffness matrices $ \stiffness_j $ and mass matrices $ \mass_{\cstate,j},
\mass_{\control,j}, \mass_{\control \cstate, j} $ can be defined as in
\eqref{eq:coarsestiffness}--\eqref{eq:coarsemass}. The construction of the  hierarchies of subspaces for controls may not be related to those for  states, as the state space and the control space may
involve completely different physical domains.
For example, the controls may be supported only on the boundary of the domain.

The final multilevel preconditioner will involve a hierarchy of matrices (and operators)  $ \multigrid_j $ each approximating  $\hessian_{j}^{-1}$. Hence, 
for the  multilevel case, the preconditioner approximates the inverse of the matrix rather than the matrix itself.
At the coarsest level, for simplicity, we assume  
$$ \multigrid_{\nlevels-1}= \hessian_{\nlevels-1}^{-1}  = \left({\solutionop^\ast_{\ell-1} \solutionop_{\ell-1} + \reg \vect{I}}\right)^{-1}.$$
The construction of $ \multigrid_{\nlevels-1}$ may be no trivial matter, since
it not only involves inverting the Hessian, but also building $\hessian_{\nlevels-1}$  by computing dense matrix-matrix products.
We will see below in Section \ref{sec:numerics} that in practice the coarse-grid problem may be approximated, but the theory here assumes an exact inverse on the coarsest level.

To define the intermediate level operators $\multigrid_j$ for $0< j<\ell-1$, 
we begin by writing the two-level preconditioner \eqref{eq:twogrid} at an arbitrary level in a multilevel hierarchy,
\begin{equation}
  \twogrid_j^{-1} = \controlinterp_j \multigrid_{j+1} \controlproject_j + \reg^{-1} (\vect{I} - \controlinterp_j \controlproject_j),
  \label{eq:twogridinmulti}
\end{equation}
where $\controlproject_j=\controlinterp_j^{\ast}$ is the matrix representation of the $L^2$-projection 
$\controlprojectop_j\in \mathfrak{L}(\controlspace_{j},\controlspace_{j+1})$.
The actual preconditioner is the first Newton iterate for the matrix equation $${\bf X}^{-1}-\hessian_{j}=\vect{0}$$ with initial guess $\twogrid_j^{-1}$,
namely
\begin{equation}
  \multigrid_{j} = 2\twogrid_j^{-1} -\twogrid_j^{-1} \hessian_j\twogrid_j^{-1}.
  \label{eq:twogridinmultiactual}
\end{equation}
At the finest level -- and this is the actual multilevel preconditioner for  $\hessian_0$ --  we define
\begin{equation}
  \multigrid_{0} = \twogrid_{0}^{-1}.
  \label{eq:twogridinmultiactualfine}
\end{equation}
Naturally, none of the operators $\multigrid_{j}$, $0\le j<\ell-1$ should be built. Instead, the action of $\vect{u}_j\leftarrow \multigrid_{j} \vect{b}_j$ 
can be easily implemented following  \cite[Algorithm MLAS]{draganescu-dupont}:
\begin{enumerate}
\item if $j=\ell-1$
\item \hspace{1cm} $\vect{u}_j\leftarrow \hessian^{-1}_j \vect{b}_j$
\item else
\item \hspace{1cm} $ \vect{u}_j \leftarrow \twogrid_j^{-1} \vect{b}_j $.
\item \hspace{1cm} if $j>0$
\item \hspace{2cm} $\vect{u}_j \leftarrow \vect{u}_j + \twogrid_j^{-1}\left( \vect{b}_j-\hessian_j\vect{u}_j\right)$
\item \hspace{1cm} end if
\item end if
\end{enumerate}
Since the action of  $\twogrid_j^{-1}$ requires the action of $\multigrid_{j+1}$, the algorithm above has a 
W-cycle structure, due to the two calls to $\twogrid_j^{-1}$ in lines 4 and 6 for intermediate levels (when $0<j<\ell-1$). 
Also at intermediate levels, the action of the Hessian $\hessian_j$
is required (line 6), and usually this is the most cost-intensive component.

It should be noted here that the number of levels involved in the multilevel preconditioner is usually not as large as the number of levels that are
in principle available from the AMG infrastructure for solving the forward problem. There is no
guarantee that the multilevel preconditioner remains positive definite, except for special circumstances (although the two-level one always is). 
However, in the presence of an aggressive coarsening strategy and three spatial dimensions, 
in practice three or four levels may often suffice to achieve a significant speedup over unpreconditioned CG.

\section{Analysis}
\label{sec:analysis}
So far we have shown that the definition of a multilevel preconditioner for the Hessian extends naturally from the geometric to the AMG context. 
The analysis follows suit to some degree, though certain details depend on the particular problem and properties of the coarsening strategy. 
In this section we prefer to focus on the operators defined in Section~\ref{ssec:multilevel} rather than their associated matrices.
Recall that the Hessian operator on $\controlspace_j$ is given by 
\begin{equation}
\label{eq:hessianoptdef}
\hessianop_j=\solutionoptrue_j^{\ast}\solutionoptrue_j+\reg \ident \in \mathfrak{L}(\controlspace_j),
\end{equation}
whose matrix representation is $\hessian_j$.
The  goal is to estimate the spectral distance between the inverse of the 
finest-level hessian $\hessianop_0^{-1}$ and the multilevel preconditioner
$\multigridop_0$ corresponding to the matrix-based definition in Section~\ref{ssec:multilevel}. The main result is
Theorem~\ref{thm:multilevelconv}.

The spectral distance between two symmetric positive definite operators $\op{X}$ and  $\op{Y}$
on a Hilbert space $\controlspace$ is defined in~\cite{draganescu-dupont} as
\begin{equation}
\label{eq:specdist}
\specdist(\op{X}, \op{Y}) = \sup_{u\in \controlspace\setminus \{0\}} \left|\ln (\op{X} u, u) - \ln (\op{Y} u, u) \right|,
\end{equation}
and is shown to be a good quality measure for the convergence of preconditioned iterative methods. 
In particular, for two symmetric positive definite operators $ \op{X}, \op{Y}$, we have 
$$ \ln \text{cond} (\op{Y}^{-1} \op{X}) \leq \specdist(\op{X}, \op{Y}),$$ and  if 
$\specdist(\hessianop_0^{-1},\multigridop_0)$ is bounded with respect to the number of levels, then the number of preconditioned conjugate gradient 
iterations is also bounded.

For $\op{X}, \op{Y}\in \mathfrak{L}(\controlspace)$ denote 
$\newt{{\mathcal X}}{{\mathcal Y}} = 2 {\mathcal Y} - {\mathcal Y}{\mathcal X}{\mathcal  Y}$, and let 
$\op{E}_j:\mathfrak{L}(\controlspace_{j+1})\to \mathfrak{L}(\controlspace_{j})$ be given by 
\begin{eqnarray*}
\op{E}_j(\op{X}) = \controlinterpop_j \op{X} \controlprojectop_j + \reg^{-1}(\ident-\controlinterpop_j\controlprojectop_j).
\end{eqnarray*}
The multilevel operator whose matrix representation is $\multigrid_j$ can be defined recursively as
\begin{eqnarray}
&&\multigridop_{\ell-1} = \hessianop_{\ell-1}^{-1},\\
&&\multigridop_{j} = \newt{\hessianop_j}{\op{E}_j(\multigridop_{j+1})}, \ \ j=1, \dots, \ell-2\\
&&\multigridop_{0} = \op{E}_0(\multigridop_{1})
\end{eqnarray}

Denote by $\mathfrak{L}_+(\controlspace)$ the set of symmetric positive definite operators on the Hilbert space $\controlspace$. 
We recall from~\cite{draganescu-dupont} a set of basic facts about the spectral distance.
\begin{theorem} The function $\specdist$ is a distance on the set of symmetric positive operators and satisfies the  following.
\begin{itemize}
\item[\textnormal{(a)}] If $\op{X}, \op{Y}\in \mathfrak{L}_+(\controlspace)$, then
  \begin{equation} 
    \label{eq:specdistinv}
    \specdist(\op{X}, \op{Y}) = \specdist(\op{X}^{-1}, \op{Y}^{-1}).
  \end{equation}
\item[\textnormal{(b)}] If $\op{X}, \op{G}\in \mathfrak{L}_+(\controlspace)$ and $\specdist(\op{X},\op{G}^{-1}) < 0.4$, then
   \begin{equation} 
    \label{eq:specdistquad}
    \specdist(\newt{\op{G}}{\op{X}},\op{G}^{-1}) \le 2 \:\specdist(\op{X}, \op{G}^{-1})^2.
  \end{equation}
\item[\textnormal{(c)}] If $\op{X}, \op{Y}\in \mathfrak{L}_+(\controlspace_{j+1})$, then $\op{E}_j(\op{X}), \op{E}_j(\op{Y})\in \mathfrak{L}_+(\controlspace_{j})$
  and 
  \begin{equation} 
    \label{eq:specdistlipschitz}
    \specdist(\op{E}_j(\op{X}),\op{E}_j(\op{Y}))\le \specdist(\op{X},\op{Y}).
  \end{equation}
\end{itemize}
\end{theorem}
%Note that this is an actual distance function also satisfying $d({\mathcal L}_1, {\mathcal L}_2) = d({\mathcal L}^{-1}_1, {\mathcal L}^{-1}_2)$.

For the analysis of our multilevel preconditioner we introduce the two-level operator 
\begin{equation}
\label{eq:twolevelopV}
\op{V}_j=\controlinterpop_j \hessianop_{j+1} \controlprojectop_j + \reg(\ident-\controlinterpop_j\controlprojectop_j)
\end{equation}
whose inverse is
\begin{equation}
\label{eq:twolevelopVinv}
\op{V}_j^{-1}=\controlinterpop_j \hessianop_{j+1}^{-1} \controlprojectop_j + \reg^{-1}(\ident-\controlinterpop_j\controlprojectop_j) = \op{E}_j(\hessianop_{j+1}^{-1}).
\end{equation}

\begin{assumption} 
\label{ass:assumptiongeneric}
We assume that the solution operators satisfy the following stability and approximation conditions:\\ There exists a level-independent constant
$\soloperatornormbound>0$ and a sequence $a_j$ with properties to be later described so that
\begin{eqnarray}
\label{eq:stabilityassumption}
&&\|\solutionoptrue_j\| \le \soloperatornormbound, \ \ j=0, 1, \dots,\ell-1\\
\label{eq:approxassumption}
&&\|\solutionoptrue_j-  \stateinterpop_j\solutionoptrue_{j+1}\controlprojectop_j\| = a_j, \ \ j=0, 1, \dots,\ell-2,
\end{eqnarray}
where for an operator $\op{L}\in \mathfrak{L}(U,V)$ with $U,V\subseteq L^2(\Omega)$
$$\|\op{L}\|=\sup_{u\in \controlspace\setminus \{0\}}\frac{\|\op{L}u\|}{\|u\|}.$$
\end{assumption}

\begin{lemma}
\label{lemma:twolevelamg}
If Assumption~\ref{ass:assumptiongeneric} holds then
\begin{eqnarray}
\label{eq:hessianapprox}
&&\|\hessianop_j- \op{V}_j\| \le 2 \soloperatornormbound a_j, \ \ j=0, 1, \dots,\ell-2.
\end{eqnarray}
Moreover, if $4 \soloperatornormbound a_j \le \reg$, then 
\begin{eqnarray}
\label{eq:hessianapproxdist}
&&\specdist(\hessianop_j,\op{V}_j) \le b_j, \ \ j=0, 1, \dots,\ell-2,
\end{eqnarray}
where
$$b_j = 4\reg^{-1} \soloperatornormbound a_j.$$
\end{lemma}
\begin{proof}
For $u\in \controlspace_j$ we have
\begin{eqnarray*}
\lefteqn{\left|\left((\hessianop_j-\op{V}_j)u,u\right)\right|}\\
 &= &
\left|\left((\solutionoptrue_j^{\ast}\solutionoptrue_j - 
\controlinterpop_j \solutionoptrue_{j+1}^{\ast}\solutionoptrue_{j+1}\controlprojectop_j)u,u\right)\right| = 
\left|\|\solutionoptrue_j u \|^2 - \|\solutionoptrue_{j+1}\controlprojectop_j u\|^2\right|\\
& =& \left|\|\solutionoptrue_j u \| - \|\solutionoptrue_{j+1}\controlprojectop_j u\|\right| \cdot
\left( \|\solutionoptrue_j u \| + \|\solutionoptrue_{j+1}\controlprojectop_j u \|\right)\\
&\le & 2  \soloperatornormbound \|\solutionoptrue_j u  -\stateinterpop_j\solutionoptrue_{j+1}\controlprojectop_j u\| \cdot \| u\|
\le 2  \soloperatornormbound a_j \|u\|^2,
\end{eqnarray*}
where we used that $\stateinterpop_j$ is an embedding and $\| \controlprojectop_j u\|\le \|u\|$. The conclusion \eqref{eq:hessianapprox} follows from the
symmetry of the operator $(\hessianop_j- \op{V}_j)$.
Note that
\begin{eqnarray}
\label{eqn:logineq}
\left|\ln x \right| \le 2 \left|x-1\right|,\ \ \mathrm{if}\ \ \left|x-1\right|\le \frac{1}{2}.
\end{eqnarray}
If $u\ne 0$ then the assumption  $4 \soloperatornormbound a_j \le \reg$ implies
\begin{eqnarray*}
\left|\frac{(\op{V}_ju,u)}{(\hessianop_j u,u)} - 1\right| = 
\left|\frac{((\op{V}_j-\hessianop_j) u,u)}{(\hessianop_j u,u)}\right| \le 
\reg^{-1} \|\op{V}_j-\hessianop_j\| \le \frac{1}{2},
\end{eqnarray*}
since $(\hessianop_j u, u) \ge \reg(u,u)$. Using~\eqref{eqn:logineq} we obtain
\begin{eqnarray}
\label{eq:ineqspecdist1}
\left|\ln\frac{(\op{V}_j u,u)}{(\hessianop_j u,u)}\right| \le 2 \left|\frac{(\op{V}_j u,u)}{(\hessianop_j u,u)} - 1\right|
\le 2\beta^{-1} \|\op{V}_j-\hessianop_j\| \le 4\beta^{-1}\soloperatornormbound a_j.
\end{eqnarray}
The conclusion \eqref{eq:hessianapproxdist} follows by taking the $\max$ in~\eqref{eq:ineqspecdist1} over $u\in \controlspace_j\setminus\{0\}$.
\end{proof}

%\begin{remark}
We should point out that in the context of geometric multigrid both state and control spaces are constructed using classical finite elements
corresponding to a sequence of meshes $\op{T}_{h_j}$; if the finer grids are obtained by, say, uniform
mesh-refinement, we have a sequence of mesh sizes $h_j=h_0 2^j$ ($h_0$ corresponds to the finest space, $h_{\ell-1}$ to the coarsest).
Under standard assumption on the elliptic equation~\eqref{eq:pdecont}, such as quasi-uniformity of the meshes and 
full elliptic regularity, and by using continuous piecewise linear elements, 
it is known  that the following approximation holds: there exists a constant $C_{mg}>0$ so that  
\begin{equation}
\label{eq:geom_twogrid_approx}
a_j\le C_{mg} h_j^2.
\end{equation}
This follows from the standard finite element a priori estimate~\cite{brenner-scott}
\begin{equation}
\label{eq:geom_conv_approx}
\|(\solutionoptrue_j- \solutionoptrue) u\| \le  C h_j^2 \|u\|,\ \ \ \forall u \in L^2(\Omega).
\end{equation}
Hence, if~\eqref{eq:geom_twogrid_approx} holds, we expect that, at least asymptotically
\begin{equation}
\label{eq:geom_twogrid_asympt}
a_{j}\approx a_{j+1}/4, \ \ \ j=0, 1,\dots, \ell-3.
\end{equation}
%\end{remark}
An approximation property as in assumption \eqref{eq:approxassumption} where the norm decreases like $ h / H $ is referred to as a ``strong approximation property'' in the algebraic multigrid literature, and generally speaking most AMG algorithms do not provide such a property, as weaker approximation is all that is necessary for two-grid convergence.
See the discussion in \cite{maclachlan-olson-theoretical-bounds}, and for an example of an AMG method with strong approximation property see \cite{hu-vassilevski-approximation}.
% \textcolor{green}{(This is probably still not a sufficient discussion of this issue.)}\\
% \textcolor{red}{[I agree, but it is probably fine to  leave it as such for the first submission, since it requires more reading. 
% I am sure reviewers will complain.]}

Hence we conduct the multigrid analysis both for the case when the approximation properties improve with resolution, as in 
the geometric multigrid case, or simply stay bounded. The precise assumption on $a_j$ (or  rather $b_j$)  will be made clear in 
Theorem~\ref{thm:multilevelconv}. We begin with a few technical results. The following lemma is closely related to 
Lemma 5.3 from~\cite{draganescu-dupont}; 
for completeness and consistency of notation we prefer to give the short technical proof.

\begin{lemma} 
\label{ass:lemmarecursion} 
Under the hypotheses and notation of Lemma~\ref{lemma:twolevelamg}, let
$d_j=\specdist(\multigridop_j,\hessianop_j^{-1})$.
If $b_j \le 0.1$ for $j=0, 1, \dots,\ell-1$, then $d_j < 0.2$ for $j=1, \dots,\ell-1$,
and the following recursion holds:
\begin{eqnarray}
\label{eq:recursionj}
d_j& \le & 2 (d_{j+1}+b_j)^2,\ \ j=0, 1, \dots,\ell-2\\
\label{eq:recursion0}
d_0& \le & d_{1} + b_0.
\end{eqnarray}
\end{lemma}
\begin{proof}
The proof proceeds by induction starting from the coarsest level, where we have $d_{\ell-1}=0$. Assume for some $j<\ell-2$ that
$d_{j+1}<0.2$. Then by~\eqref{eq:specdistlipschitz}
\begin{eqnarray*}
\specdist(\op{E}_j(\multigridop_{j+1}),\op{E}_j(\hessianop_{j+1}^{-1}))\le d_{j+1}< 0.2.
\end{eqnarray*}
If $j\ge 1$ then
\begin{eqnarray*}
d_j &=& \specdist(\multigridop_j,\hessianop_j^{-1}) = \specdist(\newt{\hessianop_j}{\op{E}_j(\multigridop_{j+1})},\hessianop_j^{-1})
\stackrel{\eqref{eq:specdistquad}}{\le} 2 (\specdist(\op{E}_j(\multigridop_{j+1}),\hessianop_j^{-1}))^2\\
&\le & 2 (\specdist(\op{E}_j(\multigridop_{j+1}),\op{V}_j^{-1}) + \specdist(\op{V}_j^{-1},\hessianop_j^{-1}))^2 \\
&\stackrel{\eqref{eq:specdistinv}}{=} &
2 (\specdist(\op{E}_j(\multigridop_{j+1}),\op{E}_j(\hessianop_{j+1}^{-1})) + \specdist(\op{V}_j,\hessianop_j))^2 \\
&\stackrel{\eqref{eq:specdistlipschitz}}{\le}& 
2 (\specdist(\multigridop_{j+1},\hessianop_{j+1}^{-1}) + \specdist(\op{V}_j,\hessianop_j))^2 
\stackrel{\eqref{eq:hessianapproxdist}}{\le} 2 (d_{j+1} + b_j)^2< 2\cdot 0.3^2 =0.18<0.2.
\end{eqnarray*}
Finally for $j=0$
\begin{eqnarray*}
d_0 &=& \specdist(\multigridop_0,\hessianop_0^{-1}) = \specdist(\op{E}_0(\multigridop_{1}),\hessianop_0^{-1})
\le  \specdist(\op{E}_0(\multigridop_{1}),\op{V}_0^{-1}) + \specdist(\op{V}_0^{-1},\hessianop_0^{-1})\\
&\stackrel{\eqref{eq:specdistinv}}{=}&
\specdist(\op{E}_0(\multigridop_{1}),\op{E}_0(\hessianop_{1}^{-1})) + \specdist(\op{V}_0,\hessianop_0) 
\stackrel{\eqref{eq:specdistlipschitz}}{\le}
\specdist(\multigridop_{1},\hessianop_{1}^{-1}) + \specdist(\op{V}_0,\hessianop_0) 
\stackrel{\eqref{eq:hessianapproxdist}}{\le} d_{1} + b_0,
\end{eqnarray*}
which completes the proof.
\end{proof}
\begin{lemma}
\label{lma:qinequality}
If $0<\alpha\le 1/8$ then the inequality 
\begin{equation}
\label{eq:qinequality}
2(\alpha \:x+1)^2\le x
\end{equation}
is satisfied by $x_{opt}=(1-4 \alpha)/(4 \alpha^2)$, and $x_{opt}>0$.
\end{lemma}
\begin{proof}
%\textcolor{red}{[Could be eliminated from paper]}
The inequality~\eqref{eq:qinequality} is equivalent to 
$$
q(x) = 2\alpha^2 x^2+(4 \alpha -1) x + 2=2(\alpha \:x+1)^2 - x \le 0.
$$
The quadratic $q$ above has a minimum at $x_{opt}=(1-4 \alpha)/(4 \alpha^2)$, and the minimal value
is 
$$
q_{\min}=-\frac{(4 \alpha -1)^2-16 \alpha^2}{8\alpha^2} = \frac{8 \alpha -1}{8\alpha^2}.
$$
Clearly $q_{\min}\le 0$ and $x_{opt}>0$ if $\alpha\le 1/8$.
\end{proof}
\begin{theorem}
\label{thm:multilevelconv}
If Assumption~\ref{ass:assumptiongeneric} holds and $b_j = 4\reg^{-1} \soloperatornormbound a_j$ 
satisfies
\begin{equation}
\label{eq:bjassumption}
b_j\le \bjconst \bjfrac^{\ell-1-j}, \ j=0, 1, \dots ,\ell-1.
\end{equation}
with $0<\bjfrac\le 1$ and $\bjconst\le \min(0.1, \bjfrac/8)$ then 
\begin{equation}
\label{eq:multilevelconv}
\specdist(\multigridop_0,\hessianop_0^{-1})\le \left(\frac{1}{4} + \bjconst\right) \bjfrac^{\ell-1}.
\end{equation}
\end{theorem}
\begin{proof}
Let $d_j=\specdist(\multigridop_j,\hessianop_j^{-1})$.
We are looking for $\Cd>0$ so that 
\begin{eqnarray}
\label{eq:djineq}
d_j\le \Cd (\bjconst \bjfrac^{\ell-1-j})^2,\ \ j=1, 2, \dots, \ell-1.
\end{eqnarray}
We perform an inductive argument from $j=\ell-1$ down to $j=1$. The estimate for the case $j=0$ will then follow. 

Since $d_{\ell-1}=0$, the case $j=\ell-1$ is trivial.
Assume that~\eqref{eq:djineq} holds for $(j+1)$ with some $1\le j\le \ell-2$.
Using~\eqref{eq:recursionj} and $\bjfrac \le 1$ we obtain
\begin{eqnarray*}
d_j& \le & 2 (d_{j+1}+b_j)^2 \le 2 (\Cd (\bjconst \bjfrac^{\ell-1-(j+1)})^2+ \bjconst \bjfrac^{\ell-1-j})^2 =
 2 (\Cd \bjconst^2 \bjfrac^{2(\ell-j-2)}+ \bjconst \bjfrac^{\ell-1-j})^2\\
&=& 2 (\Cd \bjconst \bjfrac^{\ell-j-3}+ 1)^2 (\bjconst \bjfrac^{\ell-1-j})^2 \le 2 (\Cd \bjconst \bjfrac^{-1}+ 1)^2 (\bjconst \bjfrac^{\ell-1-j})^2.
\end{eqnarray*}
We apply~\eqref{eq:qinequality} with $\alpha=\bjconst \bjfrac^{-1}$ and $x=\Cd$ to conclude that for 
$$\Cd = x_{opt}=\frac{1-4 \alpha}{4 \alpha^2} = \frac{1-4 \bjconst \bjfrac^{-1}}{4 (\bjconst \bjfrac^{-1})^2}$$
we have 
\begin{eqnarray*}
d_j& \le & \Cd (\bjconst \bjfrac^{\ell-1-j})^2,
\end{eqnarray*}
which concludes the proof by induction.
Hence, for $j=1,\dots, \ell-1$
\begin{eqnarray*}
d_j& \le & \frac{1-4 \bjconst \bjfrac^{-1}}{4 (\bjconst \bjfrac^{-1})^2} (\bjconst \bjfrac^{\ell-1-j})^2 <\frac{1}{4} \:\bjfrac^{2(\ell-j)}.
\end{eqnarray*}
Finally, by~\eqref{eq:recursion0}
\begin{eqnarray*}
d_0& \le & d_{1} + b_0 < \frac{1}{4} \:\bjfrac^{2(\ell-1)} + \bjconst \bjfrac^{\ell-1} \le \left(\frac{1}{4} \:\bjfrac^{\ell-1} + \bjconst\right) \bjfrac^{\ell-1}
\le \left(\frac{1}{4} + \bjconst\right) \bjfrac^{\ell-1}.
\end{eqnarray*}
Hence~\eqref{eq:multilevelconv} holds.
\end{proof}
A {few} remarks are in order. First, 
the hypothesis~\eqref{eq:bjassumption} on $b_j$ can be written as 
\begin{equation}
\label{eq:ajdecreaserate} 
a_j\le \reg \frac{\bjconst}{4\soloperatornormbound} \bjfrac^{\ell-1-j}, \ j=0, 1, \dots ,\ell-1.
\end{equation}
The latter form of the assumption is consistent with the approximation properties of the coarser spaces for the case 
when they represent a geometric multigrid hierarchy.
For example, the estimates~\eqref{eq:geom_conv_approx}-\eqref{eq:geom_twogrid_asympt} are consistent with $\bjfrac=1/4$. It 
is conceivable that for some cases AMG hierarchies of spaces  could also satisfy~\eqref{eq:ajdecreaserate} with $\bjfrac<1$, meaning 
that the finer spaces have better approximation properties than the coarser spaces.
The case $\bjfrac=1$ was not addressed in the earlier works on geometric multigrid, and translates into saying that there is a uniform upper bound 
for the two-level approximation as expressed in~\eqref{eq:ajdecreaserate}. In the absence of sufficient theoretical results in the AMG literature regarding
the successive two-grid approximations,
we conducted  numerical experiments in Section~\ref{ssec:aj_numerics} to verify
the relative behavior of the sequence $a_j$ for a few cases of interest.

The second remark refers to the optimality of the result in Theorem~\ref{thm:multilevelconv}. In case $\bjfrac<1$, the Theorem shows that 
$\specdist(\multigridop_0,\hessianop_0^{-1})$
decreases at the same rate as $b_0$, albeit with a larger constant in front. Therefore the quality of the 
preconditioners increases with increasing resolution, resulting in a decreasing number of  preconditioned CG iterations,
as in the case of geometric multigrid~\cite{draganescu-dupont}. 
However, if $\bjfrac=1$, then we have a uniform bound of the specral distance independent of the number of levels, which results in a mesh-independent
number of iterations.

{We also note the role of the regularization parameter $\reg$ in the multilevel convergence estimates. As might be expected, 
a small $\reg$ implies that better approximation is required from the coarse spaces, and in general the problem is harder to precondition if $\reg$ is small.}

Finally, the numbers $a_j$ can be computed numerically based on matrix norms, since 
the translation between operator and matrix norms is straightforward. 
If ${\mathcal L}\in \mathfrak{L}(\controlspace, \statespace)$ has matrix representation
${\bf L}\in \R^{\statedim\times \ctrldim}$, then
\begin{eqnarray*}
\|{\mathcal L}\| &=& \max_{u\in \controlspace}\frac{\|{\mathcal L}u\|}{\|u\|} = 
\max_{\vect{u}\in \R^{\ctrldim}}\frac{((\vect{L}\vect{u})^T\massstate\vect{L}\vect{u})^{\frac{1}{2}}}{(\vect{u}^T\masscontrol\vect{u})^{\frac{1}{2}}}
=\max_{\vect{v}\in \R^{\ctrldim}}\frac{((\massstate^{\frac{1}{2}}\vect{L}\masscontrol^{-\frac{1}{2}}\vect{v})^T\massstate^{\frac{1}{2}}\vect{L}\masscontrol^{-\frac{1}{2}}\vect{v})^{\frac{1}{2}}}{(\vect{v}^T\vect{v})^{\frac{1}{2}}} \\
&=& \|\massstate^{\frac{1}{2}}\vect{L}\masscontrol^{-\frac{1}{2}}\|_2,
\end{eqnarray*}
where we substituted $\vect{u}=\masscontrol^{-\frac{1}{2}}\vect{v}$, and $\|\cdot\|_2$ is the matrix $2$-norm.
Hence, for prototype AMG methods, one can compute the numbers $a_j$ to identify their behavior numerically by using
the formula
\begin{eqnarray}
\label{eq:ajnumerical}
a_j=\|\massstate^{\frac{1}{2}}(\solutionop_j-  \stateinterp_j\solutionop_{j+1}\controlproject_j)\masscontrol^{-\frac{1}{2}}\|_2.
\end{eqnarray}
This formula is used in Section~\ref{ssec:aj_numerics} to assess the behavior of $a_j$ for standard AMG methods.

%%%% HERE
\section{Implementation and numerical results}
\label{sec:numerics}
In this section we show numerical experiments to accompany our analysis (Section~\ref{ssec:aj_numerics}), followed in Section~\ref{ssec:numerics}
by a discussion on implementation. In Section~\ref{ssec:constant_coeff} we show numerical results for a constant coefficient case, and in 
Section~\ref{sec:geometric-comparison} we compare our results with the geometric multigrid version of this method. In Sections~\ref{ssec:complex_gem} 
and~\ref{ssec:varying_coeff} we show results for problems with complex geometries and varying coefficients, respectively.

\subsection{Numerical estimates of \eqref{eq:ajnumerical}}
\label{ssec:aj_numerics}

Since $ \massstate $ and $ \masscontrol $ can be thought of as discretizations of an identity operator, and in a finite element context 
these matrices have a spectrum that is bounded independently of the mesh size, it makes sense to approximate \eqref{eq:ajnumerical} by
\begin{equation}
  \tilde{a}_j = \| \solutionop_j-  \stateinterp_j\solutionop_{j+1}\controlproject_j \|,
  \label{eq:aj-tilde}
\end{equation}
which can be estimated effectively by approximating the dominant eigenvalue with the power method.
Tests on small matrices in 2D, where explicit matrix square roots and norms are feasible, show that $ \tilde{a}_j $ is 
quite a good approximation for $ a_j $, as shown in Table~\ref{tab:ajnumerical-2d}.
The problem here is based on the constraint \eqref{eq:pdecont} with $\coeff = 1$ on the domain $\domain = [0,1]^2$ discretized with a uniform mesh of first-order Lagrange quadrilateral elements.
Except for boundary conditions, the same discretization is used for state and control spaces.

\begin{table}[hbt!]
  \caption{Direct computation of $a_j$ coefficient from \eqref{eq:ajnumerical} and $ \bjfrac $ in \eqref{eq:ajdecreaserate} for geometric and algebraic multigrid coarsening of a uniform quadrilateral mesh.
    In this setting we can show numerically that \eqref{eq:aj-tilde} approximates \eqref{eq:ajnumerical} quite well.}
  \label{tab:ajnumerical-2d}

  \begin{center}\begin{tabular}{lccccccccc}
\toprule
   & \multicolumn{3}{c}{geometric}& \multicolumn{3}{c}{AMG(aggressive)}& \multicolumn{3}{c}{AMG(no aggressive)} \\
   \cmidrule(r){2-4} \cmidrule(r){5-7} \cmidrule(r){8-10}
  level $j$& $a_j$& $\tilde{a}_j$&  $\bjfrac$&  $a_j$& $\tilde{a}_j$ &  $\bjfrac$& $a_j$& $\tilde{a}_j$ &  $\bjfrac$ \\
\midrule
  0&  3.24e-4&  3.14e-4&   ---&  2.63e-2&  2.63e-2&  ---&   3.22e-3&  3.22e-3& --- \\
  1&  1.29e-3&  1.25e-3&  0.25&  3.86e-3&  3.86e-3&  6.81&  2.20e-3&  2.20e-3& 1.46 \\
  2&  4.95e-3&  4.80e-3&  0.26&  6.82e-3&  6.79e-3&  0.57&  7.03e-3&  6.95e-3& 0.32 \\
  3&  1.71e-2&  1.63e-2&  0.29&         &         &      &  1.60e-2&  1.60e-2& 0.43 \\
\bottomrule
  \end{tabular}\end{center}
\end{table}

In Table~\ref{tab:ajnumerical}, we use the estimate \eqref{eq:aj-tilde} for a uniform mesh of 
hexahedra on $\domain=[0,1]^3$, again with $\coeff = 1$ (the same setting used below in Section \ref{sec:geometric-comparison}).
In Tables \ref{tab:ajnumerical-2d} and \ref{tab:ajnumerical}, we see that the geometric multigrid has 
ratios of $ \bjfrac \approx \tilde{a}_j / \tilde{a}_{j+1} $ of about 1/4, as expected.
For algebraic multigrid, when aggressive coarsening is used for the first coarsening 
(as is the case in our other numerical experiments), this ratio is quite large, but otherwise it is generally below one.

% data from following two tables generated with estimate45.py
\begin{table}[hbt!]
  \caption{Numerical estimates of the $\tilde{a}_j$ coefficient from \eqref{eq:aj-tilde} and $ \bjfrac \approx \tilde{a}_j / \tilde{a}_{j+1}$ in \eqref{eq:ajdecreaserate} for geometric and algebraic multigrid coarsening of a uniform hexahedral mesh.}
  \label{tab:ajnumerical}
  \begin{center}\begin{tabular}{lcccccc}
\toprule
   & \multicolumn{2}{c}{geometric}& \multicolumn{2}{c}{AMG(aggressive)}& \multicolumn{2}{c}{AMG(no aggressive)} \\
   \cmidrule(r){2-3} \cmidrule(r){4-5} \cmidrule(r){6-7}
  level $j$& $\tilde{a}_j$&  $\tilde{a}_j / \tilde{a}_{j+1} $&  $\tilde{a}_j$ &  $\tilde{a}_j / \tilde{a}_{j+1} $& $\tilde{a}_j$ &  $\tilde{a}_j / \tilde{a}_{j+1} $ \\
\midrule
  0&  3.13e-4&   ---&  1.65e-2&  ---&   5.97e-4& --- \\
  1&  1.23e-3&  0.25&  2.16e-3&  7.64&  1.65e-3& 0.36\\
  2&  4.44e-3&  0.28&  3.58e-3&  0.60&  3.27e-3& 0.50\\
  3&  1.37e-2&  0.32&  6.92e-3&  0.52&  5.43e-3& 0.60\\
\bottomrule
  \end{tabular}\end{center}
\end{table}

\subsection{Algorithm implementation}
\label{ssec:numerics}
In practice we solve the problem 
\begin{equation}
  \left(\solutionop^T \massstate \solutionop + \reg \masscontrol\right) \vect{\control} = \solutionop^T \massstate \vect{\desired}
  \label{eq:firstorderpractice}
\end{equation}
rather than~\eqref{eq:firstorder}.
Note that the former can be obtained by multiplying the latter by $ \masscontrol $ from the left.
Then our actual implementation preconditions the operator $ \masscontrol \hessian $.
If \eqref{eq:twogrid} is a good preconditioner for $ \hessian $, then $ \twogrid^{-1} \masscontrol^{-1} $ is a 
good preconditioner for $ \masscontrol \hessian $.
With some substitutions we can write
\begin{equation}
  \twogrid^{-1} \masscontrol^{-1} = \controlinterp (\mass_{\control,H} \hessian_H)^{-1} \controlinterp^T + 
  \reg^{-1} ( \masscontrol^{-1} - \controlinterp \mass_{\control,H}^{-1} \controlinterp^T ),
\end{equation}
with analogous modifications for the multilevel preconditioner.

The AMG package we use comes from Hypre \cite{hypre-webpage,baker-falgout-scaling}, 
specifically the BoomerAMG preconditioner \cite{boomeramg}.
We use the finite element package MFEM for our finite element discretization \cite{mfem-webpage}.
Our algorithm requires the application of $ \controlproject $ and therefore the inversion of mass matrices.
We use conjugate gradient preconditioned with a symmetric Gauss-Seidel sweep to solve mass matrix problems, with a relative residual tolerance of $ 10^{-8} $.
Our emphasis in what follows is showing that the proposed algorithm works and is practical, not on tuning of parameters for maximum efficiency.

Similarly, whenever we apply the operators $ \solutionop $ or $ \solutionop^\ast $, we use conjugate gradient preconditioned with Hypre BoomerAMG to invert $\stiffness$, solving to a relative residual tolerance of $ 10^{-8} $.
The coarsest optimization solve $ \hessian_{\ell-1}^{-1} $ is done with unpreconditioned conjugate gradient and a relative residual tolerance of $ 10^{-4} $. 
We note that this same inner solver is used also when we compare our optimization preconditioner to ``unpreconditioned'' CG, that is, the inner solves in $ \solutionop, \solutionop^\ast $ always use an algebraic multigrid method.
As a result, the AMG setup cost for our preconditioner is comparable to that for the unpreconditioned case.
We note that our implementation is fully parallel, though the emphasis in this paper is not on parallel performance or scalability.

\begin{figure}
  \begin{center}
    \includegraphics[width=0.5\textwidth]{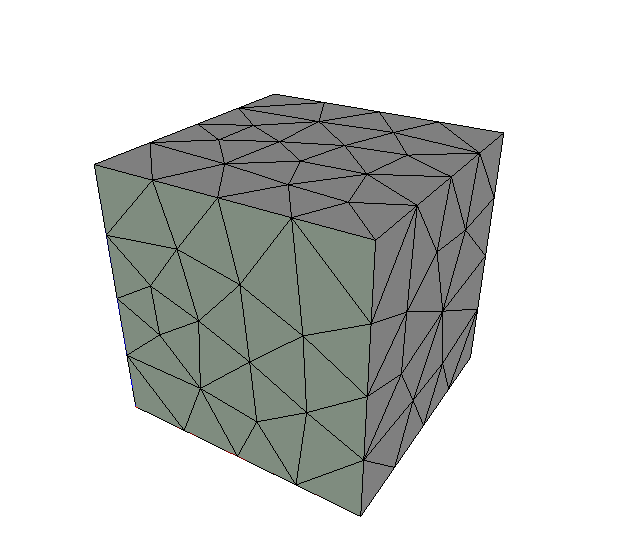}
    \includegraphics[width=0.35\textwidth]{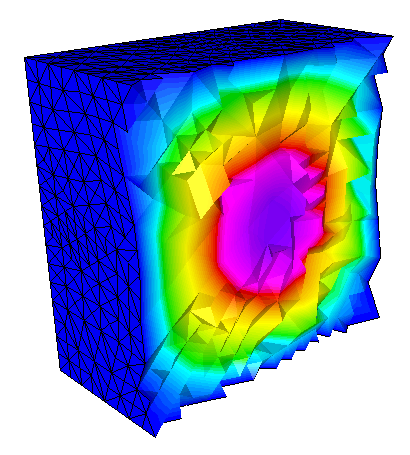}
  \end{center}
  \caption{An unstructured tetrahedral mesh used for numerical experiments (left) and a typical achieved optimal control for the problem \eqref{eq:test-state} on a refined mesh (right).}
  \label{fig:mesh}
\end{figure}

\subsection{Constant coefficient elliptic constraint}
\label{ssec:constant_coeff}
The next set of experiments discretize \eqref{eq:pdecont} with $\coeff=1$ on the domain $\domain=[0,1]^3$ and with homogeneous Dirichlet conditions on all of $ \partial \domain $.
The mesh is based on a 474-element unstructured tetrahedral mesh that is refined uniformly (see Figure \ref{fig:mesh}), and we use first-order Lagrange elements.
The desired state is set as
\begin{equation}
  \desired = \left( \frac{1}{3 \pi^2} + 3 \pi^2 \reg \right) \sin(\pi x) \sin(\pi y) \sin(\pi z)
  \label{eq:test-state}
\end{equation}
which results in a closed form solution to the optimal control problem, see Figure \ref{fig:mesh}.
In these examples all the solution methods approximate the true solution with the expected order of 
accuracy, and in particular our preconditioner does not change the solution as compared to unpreconditioned conjugate gradient.
We solve the Hessian problem \eqref{eq:firstorder} with the conjugate gradient method, stopping when the relative residual is 
less than $ 10^{-8} $, and compare use of the multilevel preconditioner $ \multigrid^{-1} $ to the unpreconditioned conjugate gradient.
In these examples we use as many levels in our multilevel Hessian preconditioner as Hypre's BoomerAMG algorithm generates for inverting 
the stiffness matrix $ \stiffness $ needed for the forward problem~\eqref{eq:solutionopdef}, as well as for the adjoint problem.
In Table \ref{tab:np8outer-iterations} we report results for the multilevel preconditioned Hessian for this problem, for 
different problem sizes $N$ (which is the number of degrees of freedom or mesh nodes) and regularization parameters~$\reg$.
By comparing Table \ref{tab:np8outer-iterations} with the corresponding results for the unpreconditioned optimization problem, Table \ref{tab:nopc_np8outer-iterations}, we see that the multilevel preconditioner provides a large efficiency gain over the unpreconditioned solver.
\begin{table}[ht!]
\caption{Number of outer conjugate gradient iterations and solve time (in seconds, in parentheses) for constant coefficient problem using the AMG-based preconditioner to the Hessian, based on the unstructured mesh in Figure \ref{fig:mesh} and running on 8 processors.}
% Generated on 31 October 2019 at 13:56 on machine quartz2300 by user barker29 with git at or near revision 8be2c00.
\label{tab:np8outer-iterations}
\begin{center}\begin{tabular}{lllll}
\toprule
  &  \multicolumn{4}{c}{$\reg$} \\
  \cmidrule(r){2-5}
$N$& 0.0001& 0.01& 1.0& 100.0 \\
\midrule
5941& 11 (.411)& 4 (.157)& 2 (.0881)& 2 (.104) \\
43881& 12 (1.57)& 4 (.571)& 2 (.319)& 2 (.368) \\
337105& 10 (6.98)& 4 (3.07)& 2 (1.79)& 2 (1.95) \\
2642337& 11 (90.7)& 4 (36.9)& 2 (21.7)& 2 (21.8) \\
\bottomrule
\end{tabular}\end{center}
\end{table}
\begin{table}[h!]
\caption{Number of unpreconditioned conjugate gradient iterations and solve time (in seconds, in parentheses) for constant coefficient Hessian problem, based on the unstructured mesh in Figure \ref{fig:mesh} running on 8 processors.}
% Generated on 31 October 2019 at 13:56 on machine quartz2300 by user barker29 with git at or near revision 8be2c00.
\label{tab:nopc_np8outer-iterations}
\begin{center}\begin{tabular}{lllll}
\toprule
  &  \multicolumn{4}{c}{$\reg$} \\
  \cmidrule(r){2-5}
$N$& 0.0001& 0.01& 1.0& 100.0 \\
\midrule
5941&  42 (.604)& 39 (.525)& 40 (.554)& 40 (.535) \\
43881& 41 (2.58)& 40 (2.54)& 41 (2.60)& 41 (2.58) \\
337105& 37 (16.3)& 40 (18.3)& 40 (17.8)& 40 (18.2) \\
2642337& 33 (211)& 40 (257)& 40 (254)& 40 (261) \\
\bottomrule
\end{tabular}\end{center}
\end{table}

\subsection{Comparison of geometric and algebraic multigrid}
\label{sec:geometric-comparison}

Here we again consider problem \eqref{eq:pdecont} with known solution \eqref{eq:test-state} and 
$ \coeff = 1 $ on $ [0, 1]^3 $, but on a structured regular hexahedral mesh where we can compare the 
algebraic multigrid approach to a geometric multigrid setting.
To more closely reflect our analysis, we use only a two-grid hierarchy here.
As expected, the results in Table \ref{tab:geometric-comparison} show that the geometric 
hierarchy has better approximation properties and faster convergence, but in this simple setting the 
AMG hierarchy also shares the key property that convergence improves as $ h \rightarrow 0 $.

% results from 31 October 2019 at 14:40 on machine quartz2300 by user barker29 with git at or near revision 4a89e25.
\begin{table}
\caption{Comparison of conjugate gradient iteration counts for geometric and algebraic hierarchies on a structured hexahedral mesh.}
\label{tab:geometric-comparison}
\begin{center}\begin{tabular}{lcccc}
\toprule
   & \multicolumn{2}{c}{$\reg=10^{-4}$}& \multicolumn{2}{c}{$\reg=10^{-6}$} \\
   \cmidrule(r){2-3} \cmidrule(r){4-5}
  $N$& AMG& geometric& AMG& geometric \\
\midrule
81& 7& 5& 24& 26\\
289& 9& 4& 40& 12\\
1089& 9& 3& 42& 7\\
4225& 7& 3& 24& 4\\
16641& 6& 3& 16& 3\\
66049& 6& 3& 13& 2\\
\bottomrule
\end{tabular}\end{center}
\end{table}

\subsection{Complex geometry}
\label{ssec:complex_gem}
\begin{figure}
\begin{center}
  \includegraphics[width=0.45\textwidth]{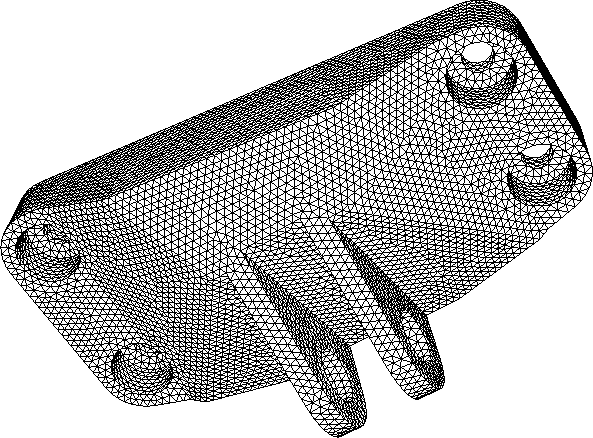}
\end{center}
\caption{A mesh for an engine bracket challenge problem \cite{morgan-ge}. We refine this complex unstructured mesh uniformly once to run our computational examples on a mesh with two million tetrahedra.}
\label{fig:ge-mesh}
\end{figure}

As an example of applying this approach to a complex geometry for which it is difficult to apply geometric multigrid techniques, we use as a test geometry a mesh for an engine bracket used for a design challenge problem in 2014 \cite{morgan-ge}.
The mesh we use for computations has two million tetrahedral elements.
For the state equation, a zero Dirichlet boundary condition is imposed on the inside surface of the two eyelets pictured at the bottom of the mesh in Figure \ref{fig:ge-mesh}, with a natural Neumann condition on the remainder of the boundary.
The desired state $ \desired = 1 $ throughout the domain, and $ \coeff = 1 $.
The optimal control when $ \reg = 1 $ is shown in Figure \ref{fig:ge-control}.
For this example, we solve on a single core and compare the results without preconditioning to a multilevel preconditioner for various regularization parameters in Table \ref{tab:bracket}, which shows some speedup for the AMG-based preconditioner.
We note that the number of levels used varies in this example---a full multilevel hierarchy is very effective for $ \reg = 1 $ but for the smaller $ \reg $ values we are restricted to only three levels, because using more levels leads to an indefinite preconditioner in the conjugate gradient solve.

\begin{figure}
\begin{center}
  \includegraphics[width=0.45\textwidth]{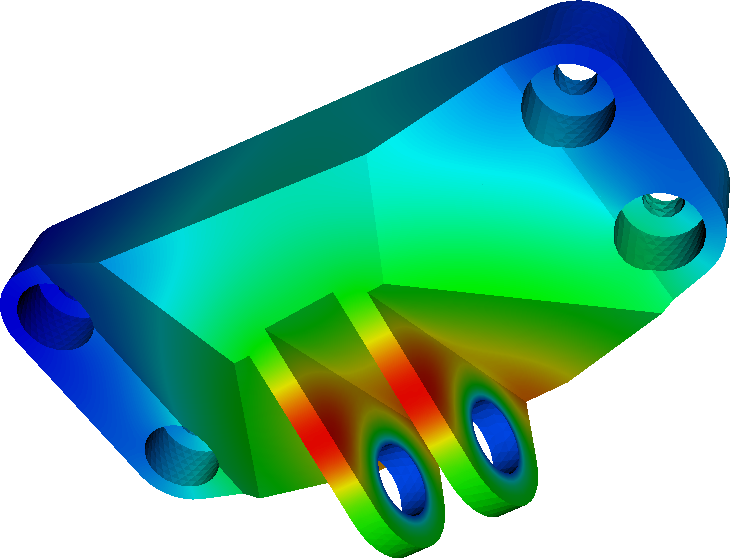}
\end{center}
\caption{Achieved optimal control for the engine bracket mesh with $ \reg = 1 $.}
\label{fig:ge-control}
\end{figure}

\begin{table}
\caption{Conjugate gradient iteration counts and solve times (in seconds) for the unstructured engine bracket example (Figure \ref{fig:ge-mesh}) for various $ \reg $ parameters.}
\label{tab:bracket}
% this table handmade, on quartz, 31 October 2019
% git hash about 67b3dbd2fed
\begin{center}\begin{tabular}{lcccr}
\toprule
  & \multicolumn{2}{c}{unpreconditioned}& \multicolumn{2}{c}{multi-level} \\
  \cmidrule(r){2-3} \cmidrule(r){4-5}
 $\reg$&  iterations&  solve time&  iterations&  solve time \\
\midrule
  1.0&    43&  206.23&  11&  64.31 \\   % five levels
  0.5&    42&  190.55&  17&  97.07 \\     % five levels
  0.25&   45&  215.32&  17&  118.83 \\     % three levels
  0.125&  50&  234.07&  21&  144.50  \\  % three levels
\bottomrule
\end{tabular}\end{center}
\end{table}

\subsection{Varying coefficient}
\label{ssec:varying_coeff}

Algebraic multigrid methods are especially attractive when the coefficient $ \coeff $ is spatially varying.
For this set of experiments we use the varying coefficient
\[
  \coeff(x) = \left\{ \begin{array}{ll} 
      1,& | x - x_c | > 1/4 \\
      \contrast,& | x - x_c | \leq 1/4,
    \end{array} \right.
\]
where $ x_c = (0.5,0.5,0.5) $ is the center of $ \domain=[0,1]^3 $, and $ \contrast $ will vary in the experiments below.
Note that the mesh does not align with the coefficient in these examples.
A zero Dirichlet condition is applied on the boundary $ z = 0 $, while the other boundaries have homogeneous Neumann conditions.
The desired state is a constant $ \desired = 1 $, and the regularization parameter is fixed at $ \reg = 1 $.

In these examples we use only the finest few levels of the multigrid hierarchy generated for the stiffness matrix 
$ \stiffness $ in our multilevel algorithm, 
because for a high contrast-coefficient $ \coeff $ (that is, for a small $\contrast$) using many levels leads to an non-convergent 
preconditioner~$ \multigrid^{-1} $.
In Table \ref{tab:ref5outer-iterations} we compare using different numbers of levels to the unpreconditioned (``none'') case.
We see again in this more challenging setting that our multilevel procedure has a fairly large efficiency advantage over unpreconditioned conjugate gradient.

\begin{table}
\caption{Number of outer conjugate gradient iterations and solve time (in seconds, in parentheses) for the varying coefficient problem for 
no preconditioning (none) and for the AMG-based preconditioner with different numbers of levels in the multilevel hierarchy. 
The original mesh is refined 5 times and this problem is run on 8 processors.}
% Generated on 31 October 2019 at 13:23 on machine quartz2300 by user barker29 with git at or near revision 8be2c00.
\label{tab:ref5outer-iterations}
\begin{center}\begin{tabular}{llllll}
\toprule
  &  & \multicolumn{4}{c}{multilevel} \\
  \cmidrule(r){3-6}
$\contrast$& \multicolumn{1}{c}{none}& \multicolumn{1}{c}{2}& \multicolumn{1}{c}{3}& \multicolumn{1}{c}{4}& \multicolumn{1}{c}{5} \\
\midrule
0.0001& 200 (1210)& 57 (768.)& ---& ---& --- \\
0.001&   87 (537.)& 16 (181.)& 15 (182.)& 15 (155.)& 16 (144.)\\
0.01&    70 (463.)& 5 (68.7)&  5 (69.0)& 5 (57.1)& 5 (50.2)\\
0.1&     70 (475.)& 3 (49.9)&  3 (47.6)& 3 (38.1)& 3 (33.0)\\
1&       70 (476.)& 3 (49.1)&  3 (47.1)& 3 (37.8)& 3 (32.7)\\
\bottomrule
\end{tabular}\end{center}
\end{table}

\section{Conclusions}
\label{sec:conclusions}
We have shown that previously developed geometric multigrid preconditioning techniques for optimal control of elliptic equations
can be successfully extended to algebraic multigrid and implemented in standard packages. The novel AMG-based preconditioner
brings a significant algorithmic efficiency for problems where geometric multigrid based preconditioning is not applicable.
In the future we expect to extend the approach to the constrained optimization case as in \cite{draganescu-petra, Draganescu_Saraswat},
and further to optimal control of semilinear elliptic equations.

\section*{Disclaimer}

This document was prepared as an account of work sponsored by an agency of the United States government. Neither the United States government nor Lawrence Livermore National Security, LLC, nor any of their employees makes any warranty, expressed or implied, or assumes any legal liability or responsibility for the accuracy, completeness, or usefulness of any information, apparatus, product, or process disclosed, or represents that its use would not infringe privately owned rights. Reference herein to any specific commercial product, process, or service by trade name, trademark, manufacturer, or otherwise does not necessarily constitute or imply its endorsement, recommendation, or favoring by the United States government or Lawrence Livermore National Security, LLC. The views and opinions of authors expressed herein do not necessarily state or reflect those of the United States government or Lawrence Livermore National Security, LLC, and shall not be used for advertising or product endorsement purposes.

\bibliography{amgopt}

\end{document}